\documentclass[proc]{edpsmath}
\usepackage{amsmath}
\usepackage{amsfonts}
\usepackage{amssymb}
\usepackage{amsthm}
\usepackage{enumerate}
\usepackage{graphicx,color,xcolor,tikz}
\usepackage{bbold}
\usepackage[T1]{fontenc}
\usepackage[utf8]{inputenc}
\usepackage{hyperref}

\newcommand{\p}{\partial}

\newcommand{\dd}{\,\mathrm{d}}

\newtheorem{lemma}[thrm]{Lemma}
\begin{document}

%%-----------------------------
%%      the top matter
%%-----------------------------
\title{Relative entropy for the numerical diffusive limit of the linear Jin-Xin system}
\author{Marianne Bessemoulin-Chatard}\address{Nantes Université, CNRS, Laboratoire de Mathématiques Jean Leray, LMJL,
UMR 6629, F-44000 Nantes, France ; \email{marianne.bessemoulin@univ-nantes.fr}}
\author{H\'el\`ene Mathis}\address{Université de Montpellier, Institut Montpellierain Alexander Grothendieck, IMAG, UMR 5149, F-34000 Montpellier, France ; \email{helene.mathis@umontpellier.fr}}

\begin{abstract} 
{This paper deals with the diffusive limit of the Jin and Xin model and its approximation by an asymptotic preserving finite volume scheme. At the continuous level, we determine a convergence rate to the diffusive limit  by means of a relative entropy method. Considering a semi-discrete approximation (discrete in space and continuous in time), we adapt the method to this semi-discrete framework and establish that the approximated solutions converge towards the discrete convection-diffusion limit with the same convergence rate.}
\end{abstract}

\maketitle

\section{Introduction}
Jin and Xin introduced in \cite{JinXin1995} a relaxation technique in order to build
robust numerical schemes for the Cauchy problem associated with the nonlinear scalar equation
\begin{equation}
\p_t u +\p_x f(u)=0,\label{eq:scal}
\end{equation}
where $f$ is typically a Lipschitz-continuous and nonlinear function.
Relaxation  consists of augmenting the equation into a
system which reads 
\begin{align}
    &\partial_t u^\nu+\partial_x v^\nu=0,\label{JX-1}\\
    &\partial_t v^\nu+\lambda^2\partial_x u^\nu=\dfrac{1}{\nu}(f(u^\nu)-v^\nu), \label{JX-2}
\end{align}
where $u^\nu,\,v^\nu :\mathbb{R}^+\times\mathbb{R}\rightarrow\mathbb{R}$ are the unknowns,
$\lambda>0$ is a given constant and $\nu$ is a relaxation parameter.
The hyperbolic part of \eqref{JX-1}-\eqref{JX-2}  is linear, of wave velocities
$\pm \lambda$. The nonlinearity of $f$ is shifted to the right hand side of the second equation.
Formally, as $\nu$ tends to zero, one observes that the second equation gives $v^\nu=f(u^\nu)$ and then solutions to \eqref{JX-1}-\eqref{JX-2} converge formally to the solution of \eqref{eq:scal}.
The question of the convergence of solutions of $u^\nu$ towards a weak entropy solution $u$ of \eqref{eq:scal} has been addressed, for example, in \cite{CCL94,Natalini96,Serre2000}.
The latter reference includes the
proof of the existence of invariant domains for the relaxed system \eqref{JX-1}-\eqref{JX-2} as long as they exist for the
the scalar equation \eqref{eq:scal}. Moreover any convex entropy $\eta\in C^2(\mathbb R)$ of the scalar equation \eqref{eq:scal} 
extends to an entropy $E\in C^2(\mathbb R^2)$ for the relaxed system \eqref{JX-1}-\eqref{JX-2} . These two ingredients
make it possible to prove, by a method of compensated compactness, 
convergence in long time to the scalar equation for arbitrary initial data.

In the present paper we focus on a slightly different scaling of the original Jin and Xin model.
The system we consider reads
\begin{align}
    &\partial_t u+\partial_x v=0,\label{JinXin-1}\\
    & \varepsilon^2\partial_t v+\lambda^2\partial_x u=au-v, \label{JinXin-2}
\end{align}
where $\varepsilon$ is the relaxation parameter. Here we consider the linear case $f(u)=au$, where $a$ is a given coefficient.
This system is endowed with an entropy-flux pair $(E,F)$ which
complies with the entropy inequality
\begin{equation}\label{ineq_EF_cont}
    \partial_t E(u,v)+\partial_x F(u,v)\leq -(au-v)^2.
\end{equation}
In the specific case of a linear relaxation, the entropy function reads
\begin{equation}\label{def_E}
    E(u,v)=\frac{\lambda^2}{2}u^2+\frac{\varepsilon^2}{2}v^2-\varepsilon^2 a u v,
\end{equation}
and the entropy flux $F$ is defined by
\begin{equation}\label{def_F}
    F(u,v)=-\frac{\lambda^2 a}{2}u^2-\frac{\varepsilon^2 a}{2}v^2+\lambda^2uv.
\end{equation}
Assuming the subcharacteristic condition \cite{Whitham74}
\begin{equation}\label{cond_sous_car}
    \lambda>\varepsilon\,|a|,
\end{equation}
the entropy functional $E(u,v)$ is strictly convex in the sense that there exists $\beta_1\geq \beta_0>0$ such that
\begin{equation}\label{E_strict_convex}
    \text{spec}(\nabla^2E)\subset [\beta_0,\beta_1].
\end{equation}
In this article, we focus on the diffusion limit of  the solutions $w=(u,v)$ of \eqref{JinXin-1}-\eqref{JinXin-2}.
Indeed in the limit $\varepsilon \to 0$, 
the solutions $w=(u,v)$ of \eqref{JinXin-1}-\eqref{JinXin-2} converge, in
a sense to be prescribed, to the solutions $\bar w=(\bar u,\bar v)$ of the
following convection-diffusion equation
\begin{align}
    & \p_t \bar u+a\p_x \bar u=\lambda^2\p_{xx}\bar u,\label{limit_1}\\
    & \bar v=a \bar u-\lambda^2\p_x \bar u. \label{limit_2}
\end{align}

In \cite{JinLiu1998}, the authors exhibited a convergence result as the
relaxation parameter $\varepsilon$ tends to zero, using a priori
estimates in appropriate Sobolev spaces, for initial data close to ones
producing a travelling-wave solution. 
Very recently Crin-Barat and Shou \cite{Crin-Barat2023} have studied the diffusive relaxation limit of
the system \eqref{JinXin-1}-\eqref{JinXin-2} 
toward viscous conservation laws \eqref{limit_1}-\eqref{limit_2} in the
multi-dimensional setting. They prove global well-posedness of strong
solutions for initial data close to constant state in suitable Besov
spaces.

Adapting the technique in \cite{Bouchut99}, Bianchini considered in
\cite{Bianchini18} the
diffusive relaxation process of the Jin-Xin model in terms of BGK type
approximations. She established the convergence toward a nonlinear heat
equation in the relaxation limit and {managed to derive a convergence rate in $O(\sqrt{\varepsilon})$ in the $L^2$ norm.} 
However, by applying a numerical scheme preserving the
hyperbolic-parabolic asymptotics proposed
in \cite{JPT98} to the model \eqref{JinXin-1}-\eqref{JinXin-2}, a rate of convergence in
$\varepsilon^2$ is observed. Figure \ref{fig:linear} presents the numerical
evidences, the details of the numerical scheme and the test case being
given in Section \ref{sec:num}.

\begin{figure}
\centering
\includegraphics[width=0.75\linewidth]{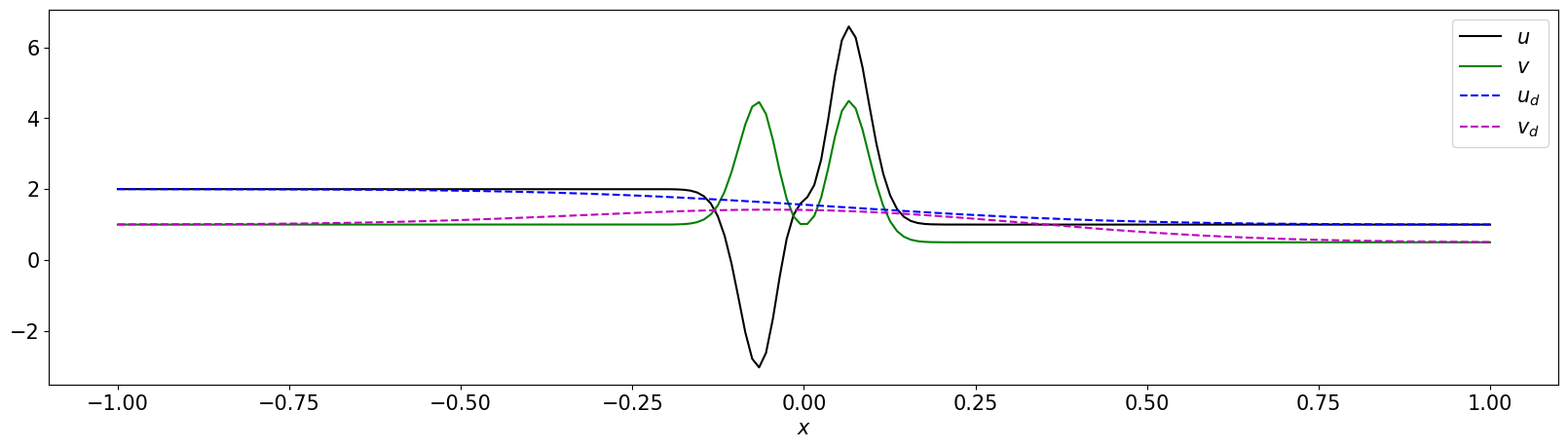}

\includegraphics[width=0.75\linewidth]{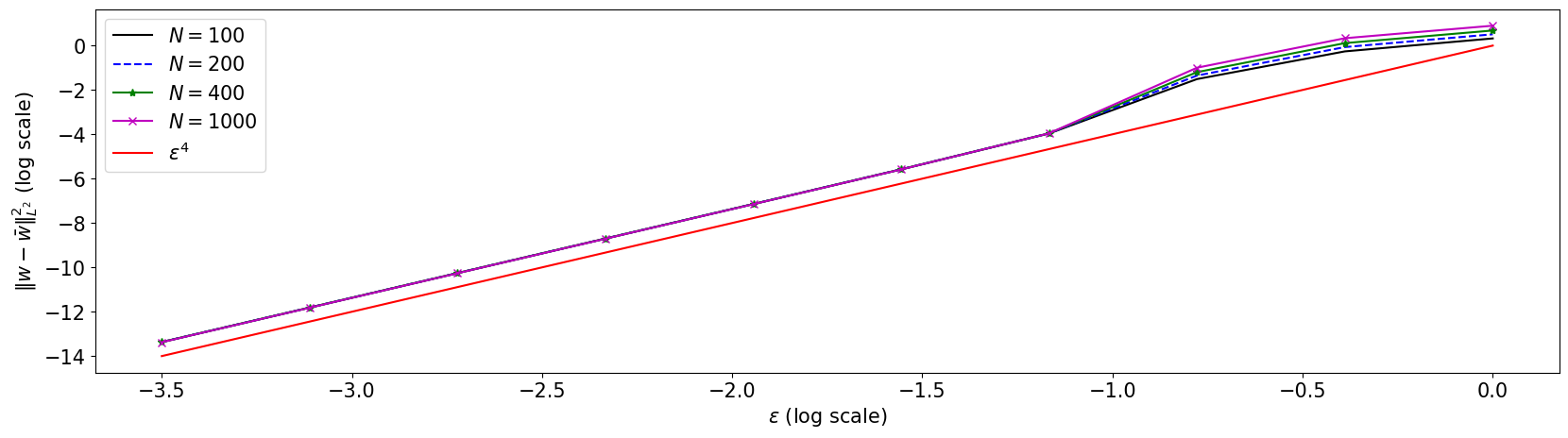}
\caption{Linear test case. Top: profile of $(u,v)$ in space, compared to $(\bar u,\bar v)$.
  Bottom: $L^2$ norm of the error $\|(u,v)-(\bar u, \bar v)\|_{L^2(Q_T)}^2$ with respect to $\varepsilon$ in log scale.}
\label{fig:linear}
\end{figure}

Our present work is motivated by the observation of this different convergence rate in numerical simulations.
It turns out that we already have observed similar numerical rates of
convergence when studying the convergence of the discrete solutions of
the $p$-system with damping towards the discrete solutions of the porous media
equation \cite{BBM16,Bulteau19}.
The numerical analysis, performed in the latter reference, is an
adaptation at the discrete level of 
the 
relative entropy method used by Lattanzio and Tzavaras
in \cite{Lattanzio2013} to prove the convergence in the continuous
setting.
The key tool is the relative entropy functional, which behaves as a squared $L^2$ norm of the difference between a
solution of the $p$-system and a solution of the porous media
equation.
Actually the relative entropy method has been applied successfully to a large
number of problems concerning hyperbolic systems, for instance: adaptation of the weak-strong uniqueness
results of Dafermos \cite{Daf79} and Di Perna \cite{DiP79} to solutions of hyperbolic systems of
conservation laws with weaker regularity
\cite{Ghoshal21, vasseur23}, uniqueness of measure-valued
solutions to hyperbolic-parabolic systems
\cite{Christoforou18,Christoforou21}, asymptotic stability of
stationary solutions to hyperbolic systems with
singular geometry terms and nonconservative products \cite{Seguin18}. Besides in \cite{Bianchini21}
Bianchini also makes use of the relative entropy method as well and 
applies it to an intermediate BGK-type model.

Our purpose is to make use of the relative entropy technique to
establish the convergence of the solutions of  the system
\eqref{JinXin-1}--\eqref{JinXin-2}  towards the solutions of
\eqref{limit_1}--\eqref{limit_2}, both at the continuous and the
discrete level. The relative entropy $E(w|\bar w)$ of the system
\eqref{JinXin-1}--\eqref{JinXin-2} is defined as the first order Taylor expansion
of $E$ around a smooth solution $\bar w$ of
\eqref{limit_1}--\eqref{limit_2}:
\begin{align}
    E(w|\bar w)&=E(w)-E(\bar w)-\nabla E(\bar w)\cdot (w-\bar w) \nonumber\\
    &= \frac{\lambda^2}{2}(u-\bar u)^2+\frac{\varepsilon^2}{2}(v-\bar v)^2-\varepsilon^2a(u-\bar u)(v-\bar v)\label{ER_cont}
\end{align}
for $w$ a classical solution of
\eqref{JinXin-1}--\eqref{JinXin-2}. Thanks to the strict convexity  \eqref{E_strict_convex} of
$E$, the relative entropy behaves like a squared $L^2$
norm of the difference between the solution $(u,v)$ of the hyperbolic
relaxation system \eqref{JinXin-1}--\eqref{JinXin-2} and the solution
$(\bar u,\bar v)$ of the convection-diffusion limit problem
\eqref{limit_1}--\eqref{limit_2}, namely
\begin{equation}\label{eq_ER_L2}
    \frac{\beta_0}{2}\left( |u-\bar u|^2+|v-\bar v|^2\right)\leq E(w|\bar w)\leq \frac{\beta_1}{2}\left( |u-\bar u|^2+|v-\bar v|^2\right).
  \end{equation}

The purpose of this article is to present similar results as in \cite{BBM16} for
convergence of solutions of \eqref{JinXin-1}--\eqref{JinXin-2} towards
solutions to \eqref{limit_1}--\eqref{limit_2}.
Section \ref{sec:continuous} is devoted to the continuous result. We first establish
a relative entropy identity. Then under
some regularity assumptions on the parabolic solutions, we establish a convergence result in relative entropy with the expected
convergence rate of $\varepsilon^2$.
The result is then extended to the semi-discrete level, by introducing a
discrete in space and continuous in time numerical scheme in Section \ref{sec:discrete}. In Section \ref{sec:rate} we construct a discrete relative entropy identity with numerical
residuals which we manage to control,  leading to the convergence
result with the expected convergence rate. The section \ref{sec:num} concludes with the details of the numerical result presented in Figure \ref{fig:linear}. In conclusion, we provide some perspectives for general nonlinear relaxation terms.

\section{The continuous setting}
\label{sec:continuous}
In this section, we first study the continuous case, adapting the diffusive relative entropy method developed in \cite{Lattanzio2013} to the case of the linear Jin-Xin relaxation system \eqref{JinXin-1}--\eqref{JinXin-2}.

We first establish an evolution law satisfied by the relative entropy \eqref{ER_cont}.

\begin{lmm}\label{lem_evolE_cont}
Let $w=(u,v)$ be a strong entropy solution of \eqref{JinXin-1}--\eqref{JinXin-2} and $\bar w=(\bar u,\bar v)$ be a smooth solution of the limit problem \eqref{limit_1}--\eqref{limit_2}. Then the relative entropy $E(w|\bar w)$, defined by \eqref{ER_cont}, satisfies the following evolution law:
\begin{equation}\label{evol_ER_cont}
    \p_t E(w|\bar w)+\p_x F(w|\bar w)=-\left(a(u-\bar u)-(v-\bar v)\right)^2+\left((v-\bar v)-a(u-\bar u)\right)\varepsilon^2\p_t \bar v,
\end{equation}
where the relative entropy flux is given by
\begin{equation}\label{FR_cont}
    F(w|\bar w)=-\frac{\lambda^2a}{2}(u-\bar u)^2-\frac{\varepsilon^2 a}{2}(v-\bar v)^2+\lambda^2(u-\bar u)(v-\bar v).
\end{equation}
\end{lmm}

\begin{proof}
Using the definition \eqref{ER_cont}, the time derivative of the relative entropy satisfies
\begin{align*}
    \p_t E(w|\bar w)=& \left[\lambda^2(u-\bar u)-\varepsilon^2 a(v-\bar v)\right]\p_t (u-\bar u)\\
    & +\left[(v-\bar v)-a(u-\bar u)\right]\varepsilon^2\p_t(v-\bar v).
\end{align*}
Remarking that the limit problem \eqref{limit_1}--\eqref{limit_2} can be written such that we get the same left hand side than for the Jin-Xin system \eqref{JinXin-1}--\eqref{JinXin-2}
\begin{align*}
    &\p_t \bar u+\p_x \bar v=0,\\
    & \varepsilon^2\p_t \bar v+\lambda^2\p_x \bar u=a \bar u-\bar v+\varepsilon^2\p_t \bar u,
\end{align*}
it yields
\begin{align*}
    \p_t E(w|\bar w)=&-\left[\lambda^2(u-\bar u)-\varepsilon^2 a(v-\bar v)\right]\p_x(v-\bar v)\\
    &-\left[(v-\bar v)-a(u-\bar u)\right]\lambda^2\p_x(u-\bar u)\\
    &-\left(a(u-\bar u)-(v-\bar v)\right)^2+\left((v-\bar v)-a(u-\bar u)\right)\varepsilon^2\p_t \bar v,
\end{align*}
which concludes the proof using the definition of the relative entropy flux \eqref{FR_cont}.
\end{proof}

In addition, we now suppose that the systems \eqref{JinXin-1}--\eqref{JinXin-2} and \eqref{limit_1}--\eqref{limit_2} are endowed with initial conditions such that the following limits hold: 
\begin{equation}
    \lim_{x\to\pm \infty}w(x,t)=\lim_{x\to\pm \infty}\bar w(x,t)=w_{\pm}, \label{CL_w}\\
\end{equation}
where $w_{\pm}$ are constant states.

Now, to compare $w$ solution of \eqref{JinXin-1}--\eqref{JinXin-2} and $\bar w$ solution of \eqref{limit_1}--\eqref{limit_2}, let us introduce the positive error estimate given by
\begin{equation}\label{phi_cont}
    \phi(t)=\int_\mathbb{R}E(w|\bar w)\,\dd x.
\end{equation}

The following convergence result, with an explicit rate, is established.

\begin{thrm}
  \label{thm:continuous}
Consider initial data $w_0$ for \eqref{JinXin-1}--\eqref{JinXin-2} and $\bar w_{0}$ for \eqref{limit_1}--\eqref{limit_2} such that $\phi(0)<+\infty$. Endowed with these initial data, let $\bar w$ be the smooth solution of \eqref{limit_1}--\eqref{limit_2} defined on $Q_T=\mathbb{R}\times[0,T)$, and $w$ be a strong entropy solution of \eqref{JinXin-1}--\eqref{JinXin-2}.
Let us assume that there exists $K>0$ such that $\| \p_t \bar v \|_{L^2(Q_T)}\leq K$, with $Q_T= \mathbb R \times [0,T)$. Then the following stability estimate holds
\begin{equation}\label{ineq_phi_cont}
    \phi(t)\leq  \phi(0)+\frac{K}{2}\varepsilon^4,   \quad t\in [0,T).
\end{equation}
Moreover, if $\phi(0)\to 0$ as $\varepsilon\to 0$, then 
\begin{equation}\label{conv_phi_cont}
  \sup_{t\in [0,T)}\phi(t)\to 0 \text{ as }\varepsilon\to 0.
\end{equation}
\end{thrm}

\begin{proof}
Using the limit assumptions \eqref{CL_w}, we first remark that $F(w|\bar w)\to 0$ in the limit $x\to\pm\infty$. Then, integrating \eqref{evol_ER_cont} on $\mathbb{R}\times [0,t)$, $t<T$, yields
\begin{align}
    \phi(t)-\phi(0)\leq & -\int_0^t\int_\mathbb{R}\left|a(u-\bar u)-(v-\bar v)\right|^2\,\dd x\dd s \nonumber\\
    & +\varepsilon^2\int_0^t\int_\mathbb{R}\left|(v-\bar v)-a(u-\bar u\right|\,|\p_t \bar v|\,\dd x\dd s.\label{estim_phi_1}
\end{align}
Concerning the last integral in this estimate, applying Cauchy-Schwarz and Young inequalities together with the assumption $\| \p_t \bar v \|_{L^2(Q_T)}\leq K$, we obtain
\begin{align*}
    \varepsilon^2\int_0^t\int_\mathbb{R}\left|(v-\bar v)-a(u-\bar u\right|\,|\p_t \bar v|\,\dd x\dd s &\leq \frac{1}{2}\int_0^t\int_\mathbb{R}\left|a(u-\bar u)-(v-\bar v)\right|^2\,\dd x\dd s\\
    & \quad +\frac{\varepsilon^4}{2}\int_0^t\int_\mathbb{R}|\p_t \bar v|^2\,\dd x\dd s\\
    & \leq \frac{1}{2}\int_0^t\int_\mathbb{R}\left|a(u-\bar u)-(v-\bar v)\right|^2\,\dd x\dd s +\frac{K}{2}\varepsilon^4.
\end{align*}
Then, inequality \eqref{estim_phi_1} becomes
\[ \phi(t)-\phi(0)\leq -\frac{1}{2}\int_0^t\int_\mathbb{R}\left|a(u-\bar u)-(v-\bar v)\right|^2\,\dd x\dd s +\frac{K}{2}\varepsilon^4, \]
which concludes the proof.
\end{proof}

\section{Semi-discrete finite volume scheme and numerical convergence rate}
\label{sec:discrete}
From a numerical point of view, the key ingredient is to consider a
numerical scheme for \eqref{JinXin-1}-\eqref{JinXin-2}
which provides the required
discretization of \eqref{limit_1}-\eqref{limit_2} in the limit of $\varepsilon$ to
zero.

Such schemes refer to \textit{Asymptotic Preserving Schemes}, notion
introduced by Jin in \cite{JinAP}. Such schemes 1) have to provide a
consistent discretization of the hyperbolic solutions of
\eqref{JinXin-1}-\eqref{JinXin-2} and
of the parabolic solutions of \eqref{limit_1}-\eqref{limit_2} at the
limit $\varepsilon \to 0$, 2) admit a CFL condition which  does not
degenerate as $\varepsilon\to 0$.

In the following, we focus on a scheme which is continuous in time and
discrete in space. Therefore the point 2) is not restrictive here.
However the numerical results given in the Introduction, see Figure
\ref{fig:linear}, have been obtained by an asymptotic preserving scheme, satisfying both
1) and 2), introduced in \cite{JPT98} and defined in
Section \ref{sec:num}.
Note that the relative entropy method has been applied to a full discrete scheme
by
Bulteau \textit{et al} \cite{Bulteau19}, for the $p$-system asymptotic
towards the porous media equation.

\subsection{Definition of the semi-discrete scheme}
\label{sec:scheme}
Let us consider a uniform mesh made of cells
$(x_{i-\frac{1}{2}},x_{i+\frac{1}{2}})_{i\in\mathbb Z}$ with uniform
size step $\Delta
x$. Here, the centers of cells are denoted $x_i=i\Delta x$ for
all $i\in\mathbb Z$. On each cell $(x_{i-\frac{1}{2}},x_{i+\frac{1}{2}})$,
the solutions of \eqref{JinXin-1}-\eqref{JinXin-2} are approximated by
time dependent
piecewise constant functions $w_i(t)={}^t(u_i(t),v_i(t))$.
The space discretization scheme is based
on the standard HLL numerical fluxes (see \cite{HLL}).
Hence the continuous in time and discrete in space numerical scheme reads
\begin{equation}
  \label{eq:Hd}
\left\{  \begin{aligned}
   \dfrac{d}{dt} u_i &= -\dfrac{1}{2\Delta x} \left(
      v_{i+1} - v_{i-1}\right) + \dfrac{\lambda}{2\Delta x}(u_{i+1}-2u_i+u_{i-1}
      ),\\
   \dfrac{d}{dt}v_i &=-\dfrac{\lambda^2}{2\varepsilon^2\Delta x}(u_{i+1}-u_{i-1})
    + \dfrac{\lambda}{2\Delta x}(v_{i+1} -
   2v_i+v_{i+1}) +\dfrac{1}{\varepsilon^2}(au_i-v_i).
  \end{aligned}\right.
\end{equation}

As soon as $\varepsilon$ goes to zero,
this finite volume scheme provides a consistent approximation of the
parabolic limit \eqref{limit_1}-\eqref{limit_2}: the pair ${\bar w}_i(t)={}^t(\bar u_i(t),\bar v_i(t))$ evolves in time as follows
 \begin{equation}
      \label{eq:Pd}
\left\{      \begin{aligned}
    &\dfrac{d}{dt} \bar u_i = -\dfrac{1}{2\Delta x} (
          \bar v_{i+1} - \bar v_{i-1} )+ \dfrac{\lambda}{2\,\Delta
            x}(\bar u_{i+1}-2 \bar u_i+\bar u_{i-1}
          ),\\
        &\dfrac{\lambda^2}{2 \Delta x} ( \bar u_{i+1}-\bar u_{i-1})
        =a\bar u_i - \bar v_i.
      \end{aligned}\right.
    \end{equation}

The numerical scheme is endowed with convenient 
limit conditions, in agreement with \eqref{CL_w} to be imposed to the approximate
solution as follows:
\begin{equation}
  \label{eq:CI_dis}
  \begin{aligned}
    \lim_{i\to \pm \infty} u_i &= \lim_{i\to \pm \infty} \bar u_i
    = u_{\pm},\\
    \lim_{i\to \pm \infty} v_i &= \lim_{i\to \pm \infty} \bar v_i
    =  v_{\pm}.
  \end{aligned}
\end{equation}

Finally, to simplify the forthcoming computations, we introduce some notations. Let $w(t)=(w_i(t))_{i\in\mathbb{Z}}$ be a function of time $t\in [0,T)$, piecewise constant on cells $(x_{i-\frac{1}{2}},x_{i+\frac{1}{2}})$. Then we define for $i\in\mathbb{Z}$
\begin{equation*}
    (D_xw)_{i+\frac{1}{2}}=\frac{u_{i+1}-u_{i}}{\Delta x},\quad (D_{xx}w)_i=\frac{w_{i+1}-2w_i+w_{i-1}}{\Delta x^2}.
\end{equation*}
We also introduce the following norm:
\[ \|w\|_{L^2(Q_t)}=\left(\int_0^t\sum_{i\in\mathbb{Z}}\Delta x\,|w_i(s)|^2\dd s\right)^{1/2}.\]

\subsection{Convergence rate}
\label{sec:rate}
We adapt the discrete relative entropy method of \cite{BBM16},
inspired of the continuous approach introduced by Lattanzio and Tzavaras
\cite{Lattanzio2013}, to the semi-discrete scheme (\ref{eq:Hd}).

First, according to the definition of the relative entropy given
by \eqref{ER_cont}, we define the discrete relative entropy function by
\begin{equation}    
  \label{eq:ERD}
  \begin{aligned}
    E_i(t)&=E(u_i, v_i|
     \bar u_i, \bar v_i)(t)\\
    &=\dfrac{\lambda^2}{2}(u_i(t)-\bar u_i(t))^2 +
    \dfrac{\varepsilon^2}{2}(v_i(t)-\bar v_i(t))^2 - \varepsilon^2 a
    (u_i(t)-\bar u_i(t))(v_i(t)-\bar v_i(t)).
  \end{aligned}
\end{equation}

Mimicking the continuous framework, we introduce $\phi(t)$
to denote the discrete space
integral of $E_i(t)$ as follows:
\begin{equation}
  \label{eq:phieddef}
  \phi(t) = \sum_{i \in \mathbb Z} \Delta x\; E_i(t).
\end{equation}
Without ambiguity and for the sake of clarity, the time dependence is omitted in the sequel.

Now, we state the discrete counterpart of Theorem \ref{thm:continuous}.
\begin{thrm}
  \label{thm:2}
  Let $\bar w_i(t)=(\bar u_i(t), \bar v_i(t))_{i \in \mathbb Z}$ be
    a smooth solution of \eqref{limit_1}-\eqref{limit_2} defined on $Q_T= \mathbb R
    \times [0,T)$. We assume the existence of a positive constant
      $K<+\infty$ such that the following estimations are satisfied:
\begin{equation}\label{hyp_vbar_dis}
    \left\|\frac{\dd}{\dd t}\bar v\right\|_{L^2(Q_T)}\leq K,\qquad \|D_{xx}\bar v\|_{L^2(Q_T)}\leq K.
\end{equation}
      Let $w_i(t)=( u_i(t),  v_i(t))_{i \in \mathbb Z}$ be a solution of
\eqref{eq:Hd} such that $\phi^\varepsilon(0) <+\infty$. We assume that the subcharacteristic condition \eqref{cond_sous_car} is fulfilled, as well as the assumptions on the limit conditions \eqref{eq:CI_dis}.
Then we have
    \begin{equation}
      \label{eq:rateD}
      \phi(t) \leq \phi(0) +
      B\,\varepsilon^4, \quad t\in [0,T),
    \end{equation}
where $B$ is a positive constant which depends only on $\lambda$ and $K$.
Moreover if $\phi(0)
\to 0$ as $\varepsilon\to 0$ then $\sup_{t\in[0,T)}\phi(t)
  \to 0$ when $\varepsilon\to 0$.
\end{thrm}

{Because the relative entropy behaves like a squared $L^2$ norm of $w-\bar w$, see \eqref{eq_ER_L2}, it follows from \eqref{eq:rateD} that the convergence rate of $\|w-\bar w\|^2_{L^2(Q_T)}$ behaves like $O(\varepsilon^4)$.}

To establish this convergence result, the first step is to 
exhibit the evolution equation satisfied by the relative entropy
$E_i$, namely the discrete counterpart of Lemma \ref{lem_evolE_cont}.

\begin{lemma} \label{lem_evolE_dis}
Let $(\bar u_i, \bar v_i)_{i \in \mathbb Z}$ be a smooth solution
of \eqref{eq:Pd} and let $(\tau_i,u_i)_{i \in \mathbb Z}$ be a
solution of \eqref{eq:Hd}. The relative entropy $E_i$,
defined by (\ref{eq:ERD}), verifies the following evolution law:
  \begin{equation}
    \label{eq:dtERD}
    \begin{aligned}
      \dfrac{\dd E_i}{\dd t} +\dfrac{1}{\Delta
        x}(F_{i+1/2}&-F_{i-1/2})=
      - \left[ a(u_i-\bar u_i) - (v_i-\bar v_i)\right]^2
      + \varepsilon^2 \left[a(u_i-\bar u_i)-(v_i-\bar v_i)
      \right]\dfrac{\mathrm d }{\mathrm d t} \bar v_i\\
      &+ R_i^1 + R_i^2+R_i^3+R_i^4,
    \end{aligned}
  \end{equation}
where $F_{i+1/2}$ corresponds to an approximation of the relative
entropy flux $F(w|\bar w)$ at the interface $x_{i+1/2}$ given by
\begin{equation}
  \label{eq:FD}
  \begin{aligned}
    F_{i+1/2} &=-\dfrac{\varepsilon^2}{2}a (v_i-\bar v_i)(v_{i+1} -
    \bar v_{i+1}) - \dfrac{\lambda^2}{2}a (u_{i} - \bar u_{i})
    (u_{i+1}-\bar u_{i+1}) \\
    &+ \dfrac{\lambda^2}{2} \left[ (u_i-\bar
      u_i) (v_{i+1}-\bar v_{i+1})
      + (u_{i+1}-\bar u_{i+1})(v_i-\bar
      v_i)\right],
  \end{aligned}
\end{equation}
and the quantities $R_i^j$, $j=1,\ldots,4$ denote numerical residuals given by
\begin{equation}
  \label{eq:Ri}
  \begin{aligned}
    R_i^1&=\frac{\lambda^3}{2}\Delta x(u_i-\bar u_i)\left(D_{xx}(u-\bar u)\right)_i,\\
    R_i^2&=\frac{\varepsilon^2\lambda}{2}\,\Delta x\, (v_i-\bar v_i)\left(D_{xx}(v-\bar v)\right)_i,\\
    R_i^3&=\frac{\varepsilon^2\lambda}{2}\,\Delta x\,\left[(v_i-\bar v_i)-a(u_i-\bar u_i)\right](D_{xx}\bar v)_i,\\
    R_i^4&=-\varepsilon^2\,a\frac{\lambda}{2}\,\Delta x\,\left[ (v_i-\bar v_i)\left(D_{xx}(u-\bar u)\right)_i+(u_i-\bar u_i)\left(D_{xx}(v-\bar v)\right)_i \right].
  \end{aligned}
\end{equation}
\end{lemma}
Observe that this evolution equation turns out to
be a discrete form of \eqref{evol_ER_cont} supplemented by
numerical viscosity terms.

\begin{proof}
According to the definition \eqref{eq:ERD}, the time derivative of the semi-discrete relative entropy $E_i$ reads
\begin{align}
   \frac{\dd E_i}{\dd t}=&\lambda^2(u_i-\bar u_i)\frac{\dd}{\dd t}(u_i-\bar u_i)+\varepsilon^2(v_i-\bar v_i)\frac{\dd}{\dd t}(v_i-\bar v_i) \nonumber \\
   &-\varepsilon^2a\frac{\dd}{\dd t}(u_i-\bar u_i)(v_i-\bar v_i)-\varepsilon^2 a (u_i-\bar u_i) \frac{\dd}{\dd t}(v_i-\bar v_i).\label{estim1Ed}
\end{align}
Now, we rewrite the second equation of the scheme \eqref{eq:Pd} as follows:
\[ \varepsilon^2\frac{\dd}{\dd t}\bar v_i=-\frac{\lambda^2}{2\Delta x}(\bar u_{i+1}-\bar u_{i-1})+(a\bar u_i-\bar v_i)+\varepsilon^2\frac{\dd}{\dd t}\bar v_i. \]
Using \eqref{eq:Hd}, we obtain
\begin{align*}
    \varepsilon^2\frac{\dd}{\dd t}(v_i-\bar v_i)=&-\frac{\lambda^2}{2\Delta x}\left[(u_{i+1}-\bar u_{i+1})-(u_{i-1}-\bar u_{i-1})\right]+\varepsilon^2\frac{\lambda}{2}\,\Delta x\,(D_{xx}v)_i\\
    &+\left[ a(u_i-\bar u_i)-(v_i-\bar v_i) \right]-\varepsilon^2\frac{\dd}{\dd t}\bar v_i.
\end{align*}
Plugging this equality in \eqref{estim1Ed} and using the first equations of \eqref{eq:Hd} and \eqref{eq:Pd} lead to
\begin{align*}
    \frac{\dd E_i}{\dd t}=&-\lambda^2(u_i-\bar u_i)\frac{1}{2\Delta x}\left[ (v_{i+1}-\bar v_{i+1})-(v_{i-1}-\bar v_{i-1}) \right]+\lambda^2(u_i-\bar u_i)\frac{\lambda}{2}\Delta x\left(D_{xx}(u-\bar u)\right)_i\\
    &-(v_i-\bar v_i)\frac{\lambda^2}{2\Delta x}\left[ (u_{i+1}-\bar u_{i+1})-(u_{i-1}-\bar u_{i-1}) \right]+(v_i-\bar v_i)\frac{\varepsilon^2\lambda}{2}\Delta x\left(D_{xx}v\right)_i\\
    &+(v_i-\bar v_i)\left[a(u_i-\bar u_i)-(v_i-\bar v_i)\right]-(v_i-\bar v_i)\varepsilon^2\frac{\dd}{\dd t}\bar v_i\\
    &+\varepsilon^2a(v_i-\bar v_i)\frac{1}{2\Delta x}\left[ (v_{i+1}-\bar v_{i+1})-(v_{i-1}-\bar v_{i-1}) \right]-\varepsilon^2a(v_i-\bar v_i)\frac{\lambda}{2}\Delta x\left(D_{xx}(u-\bar u)\right)_i\\
    &+a(u_i-\bar u_i)\frac{\lambda^2}{2\Delta x}\left[ (u_{i+1}-\bar u_{i+1})-(u_{i-1}-\bar u_{i-1}) \right]-\varepsilon^2 a (u_i-\bar u_i)\left(D_{xx}v\right)_i\\
    &-a(u_i-\bar u_i)\left[a(u_i-\bar u_i)-(v_i-\bar v_i)\right]+a(u_i-\bar u_i)\varepsilon^2\frac{\dd}{\dd t}\bar v_i.
\end{align*}
By rearranging the terms and using the definition \eqref{eq:FD} of the relative entropy flux, it yields
\begin{gather*}
    \frac{\dd E_i}{\dd t}+\frac{1}{\Delta x}(F_{i+\frac{1}{2}}-F_{i-\frac{1}{2}})=-\left[a(u_i-\bar u_i)-(v_i-\bar v_i)\right]^2+\left[a(u_i-\bar u_i)-(v_i-\bar v_i)\right]\varepsilon^2\frac{\dd}{\dd t}\bar v_i\\
    +\lambda^2(u_i-\bar u_i)\frac{\lambda}{2}\Delta x\left(D_{xx}(u-\bar u)\right)_i+(v_i-\bar v_i)\frac{\varepsilon^2\lambda}{2}\Delta x\left(D_{xx}v\right)_i\\
    -\varepsilon^2a(v_i-\bar v_i)\frac{\lambda}{2}\Delta x\left(D_{xx}(u-\bar u)\right)_i-\varepsilon^2a(u_i-\bar u_i)\frac{\lambda}{2}\Delta x\left(D_{xx}v\right)_i.
\end{gather*}
Writing $D_{xx}v$ as $D_{xx}(v-\bar v)+D_{xx}\bar v$, we are now able to identify the remainder terms $R_i^j$, $j=1,\ldots,4$, which concludes the proof.
\end{proof}
From now on, we state estimations satisfied by residuals $R_i^j$, $j=1,\ldots,4$.

\begin{lemma}\label{lem:estimR} 
Under the subcharacteristic condition \eqref{cond_sous_car}, for all $\theta>0$, we have the following estimates
\begin{align}
    \int_0^t\sum_{i\in\mathbb Z}\Delta x R_i^1\dd s&= -\frac{\lambda^3}{2}\Delta x\,\|D_x(u-\bar u)\|_{L^2(Q_t)}^2, \label{estimR1} \\
    \int_0^t\sum_{i\in\mathbb Z}\Delta x R_i^2\dd s&=-\frac{\varepsilon^2\lambda}{2}\Delta x\,\|D_x(v-\bar v)\|_{L^2(Q_t)}^2,\label{estimR2} \\
    \int_0^t\sum_{i\in\mathbb Z}\Delta x R_i^3\dd s&\leq \varepsilon^4\frac{\lambda^2}{8\theta}\Delta x^2\|D_{xx}\bar v\|_{L^2(Q_t)}^2+\frac{\theta}{2}\int_0^t\sum_{i\in\mathbb Z}\Delta x\left[(v_i-\bar v_i)-a(u_i-\bar u_i)\right]^2\dd s, \label{estimR3}\\
    \int_0^t\sum_{i\in\mathbb Z}\Delta x R_i^4\dd s&\leq \frac{\lambda^3}{2}\Delta x\,\|D_x(u-\bar u)\|_{L^2(Q_t)}^2+\frac{\lambda}{2}\varepsilon^2\,\Delta x\,\|D_x(v-\bar v)\|^2_{L^2(Q_t)}. \label{estimR4}
\end{align}
\end{lemma}

\begin{proof}
Equality \eqref{estimR1} (resp. \eqref{estimR2}) is directly obtained by summing $R_i^1$ (resp. $R_i^2$) over $i\in\mathbb Z$, applying a discrete integration by parts, and integrating with respect to time.

Then, by definition of $R_i^3$ \eqref{eq:Ri}, summing over $i\in\mathbb Z$, integrating with respect to $t$ and using the Cauchy-Schwarz inequality, we obtain
\begin{equation*}
    \int_0^t\sum_{i\in\mathbb Z}\Delta x R_i^3\dd s \leq \frac{\varepsilon^2\lambda}{2}\Delta x \left( \int_0^t\sum_{i\in\mathbb Z}\Delta x\left[ (v_i-\bar v_i)-a(u_i-\bar u_i) \right]^2\dd s \right)^{1/2}\,\|D_{xx}\bar v\|_{L^2(Q_t)}.
\end{equation*}
Finally, we apply the Young inequality with $\theta>0$ to get \eqref{estimR3}.

At last, we prove estimate \eqref{estimR4}. To do this, we first perform a discrete integration by parts and apply the Cauchy-Schwarz inequality to obtain
\begin{align*}
    \int_0^t\sum_{i\in\mathbb Z}\Delta x R_i^4\dd s &=  \varepsilon^2a\,\lambda\,\Delta x\int_0^t\sum_{i\in\mathbb Z}\Delta x \left(D_x(u-\bar u) \right)_{i+\frac{1}{2}}\left(D_x(v-\bar v) \right)_{i+\frac{1}{2}}\,\dd s\\
    &\leq \varepsilon^2 |a|\lambda\,\Delta x\,\|D_x(u-\bar u)\|_{L^2(Q_t)}\,\|D_x(v-\bar v)\|_{L^2(Q_t)}.
\end{align*}
But using the subcharacteristic condition, we have $\varepsilon |a|<\lambda$, which yields
\begin{equation*}
    \int_0^t\sum_{i\in\mathbb Z}\Delta x R_i^4\dd s\leq \lambda^{3/2}\sqrt{\Delta x}\,\|D_x(u-\bar u)\|_{L^2(Q_t)}\lambda^{1/2}\varepsilon\sqrt{\Delta x}\|D_x(v-\bar v)\|_{L^2(Q_t)},
\end{equation*}
and we finally get \eqref{estimR4} thanks to Young inequality.
\end{proof}

With these results, we can now establish the proof of Theorem \ref{thm:2}.

\begin{proof}[Proof of Theorem \ref{thm:2}]
First of all, let us remark that thanks to Lemma \ref{lem:estimR}, we
have 
\[\int_0^t\sum_{i\in\mathbb Z}\Delta x (R_i^1+R_i^2+R_i^4)\dd s \leq 0.\]
Then, integrating with respect to $t$ and summing for $i\in\mathbb Z$ the semi-discrete evolution law derived in Lemma \ref{lem_evolE_dis}, and using the estimate \eqref{estimR3}, we obtain
\begin{align*}
    \phi(t)-\phi(0) \leq & -\int_0^t\sum_{i\in\mathbb Z}\Delta x \left[ a(u_i-\bar u_i)-(v_i-\bar v_i)\right]^2\dd s+\varepsilon^2\int_0^t\sum_{i\in\mathbb Z}\Delta x \left[ a(u_i-\bar u_i)-(v_i-\bar v_i)\right]\frac{\dd \bar v_i}{\dd t}\dd s\\
    & +\varepsilon^4\frac{\lambda^2}{8\theta}\Delta x^2\|D_{xx}\bar v\|_{L^2(Q_t)}^2+\frac{\theta}{2}\int_0^t\sum_{i\in\mathbb Z}\Delta x\left[(v_i-\bar v_i)-a(u_i-\bar u_i)\right]^2\dd s.
\end{align*}
Applying Cauchy-Schwarz et Young inequalities with $\alpha>0$ on the second term of the right hand side, we get
\begin{align*}
    \phi(t)-\phi(0) \leq & \left(-1+\frac{\alpha}{2}+\frac{\theta}{2} \right)\int_0^t\sum_{i\in\mathbb Z}\Delta x \left[ a(u_i-\bar u_i)-(v_i-\bar v_i)\right]^2\dd s\\
    & + \varepsilon^4\left(\left\|\frac{\dd}{\dd t}\bar v\right\|^2_{L^2(Q_t)}+\frac{\lambda^2}{8\theta}\Delta x^2\|D_{xx}\bar v\|_{L^2(Q_t)}^2 \right).
\end{align*}
Choosing, for example, $\alpha=\frac{1}{2}=\theta$ concludes the proof thanks to assumptions \eqref{hyp_vbar_dis}.
\end{proof}

\subsection{Numerical experiments}
\label{sec:num}
The numerical results presented in Figure \ref{fig:linear} have been
obtained by a fully
discrete scheme as proposed by Jin, Pareschi and Toscani \cite{JPT98}. This
scheme is based on a reformulation of system \eqref{JinXin-1}-\eqref{JinXin-2}
\begin{equation*}\left\{\begin{aligned}
& \p_{t}u-\p_{x}v=0,\\
& \displaystyle{\p_{t}v+\lambda^2\p_{x}u=-\frac{1}{\varepsilon^{2}}\left(au-v-(1-\varepsilon^{2})\lambda^2\p_{x}u\right)}.
\end{aligned}\right.
\end{equation*}
This reformulated system is approximated by a 2-step splitting
technique.
In the first step,  the convective and non-stiff
system is approximated thanks to a classical HLL scheme
\cite{HLL}:
\begin{subequations}\label{scheme-JPT}
\begin{align}
u_{i}^{n+\frac{1}{2}}=u_{i}^{n}-\frac{\Delta t}{\Delta x}\left(\mathcal{F}_{i+\frac{1}{2}}^{u}-\mathcal{F}_{i-\frac{1}{2}}^{u}\right),\label{scheme-JPT-u-1}\\
v_{i}^{n+\frac{1}{2}}=v_{i}^{n}-\frac{\Delta t}{\Delta
  x}\left(\mathcal{F}_{i+\frac{1}{2}}^{v}-\mathcal{F}_{i-\frac{1}{2}}^{v}\right),\label{scheme-JPT-v-1}
\end{align}
\end{subequations}
where the numerical fluxes are defined by
\begin{align*}
&\mathcal{F}_{i+\frac{1}{2}}^{u}=\frac{1}{2}(v_{i}^{n}+v_{i+1}^{n})-\frac{\lambda}{2}(u_{i+1}^{n}-u_{i}^{n}),\\
&\mathcal{F}_{i+\frac{1}{2}}^{v}=\frac{\lambda^2}{2}(u_{i}^{n}+u_{i+1}^{n})-\frac{\lambda}{2}(v_{i+1}^{n}-v_{i}^{n}).
\end{align*}
This scheme is stable under the CFL condition
$\displaystyle{\frac{\Delta t}{\Delta x}\lambda\leq \frac{1}{2}}$ which does
not depend on $\varepsilon$. Next, the stiff source term is treated
in an implicit way
to obtain unconditional stability:
\begin{align*}
&u_{i}^{n+1}=u_{i}^{n+\frac{1}{2}},\\
& \frac{v_{i}^{n+1}-v_{i}^{n+\frac{1}{2}}}{\Delta t}=\frac{1}{\varepsilon^{2}}\left(au_{i}^{n+1}-v_{i}^{n+1}-(1-\varepsilon^{2})\lambda^2\frac{u_{i+\frac{1}{2}}^{n+1}-u_{i-\frac{1}{2}}^{n+1}}{\Delta x}\right),
\end{align*}
where we take $u_{i+\frac{1}{2}}^{n+1} = \dfrac{u_{i+1}^{n+1}+u_i^{n+1}}{2}$.

Since $u_{i}^{n+1}=u_{i}^{n+\frac{1}{2}}$, let us emphasize that
$v_{i}^{n+1}$ can be computed explicitly from
$(u_{i}^{n},v_{i}^{n})_{i\in\mathbb{Z}}$. Finally, the relaxation
step can be written as
\begin{subequations}\label{scheme-JPT2}
\begin{align}
&u_{i}^{n+1}=u_{i}^{n+\frac{1}{2}},\label{scheme-JPT-tau-2}\\
& v_{i}^{n+1}=\left(\frac{\varepsilon^{2}}{\varepsilon^{2}+\Delta
                                                                t}\right)v_{i}^{n+\frac{1}{2}}
                                                                +\dfrac{\Delta t}{\varepsilon^{2}+\Delta
                                                                t}\left(
                                                                au_{i}^{n+\frac{1}{2}}
                                                                -
                                                                (1-\varepsilon^2)\dfrac{\lambda^2}{2
                                                                \Delta
                                                                x}(u_{i+1}^{n+\frac{1}{2}}-u_{i-1}^{n+\frac{1}{2}}) \right). \label{scheme-JPT-u-2}
\end{align}
\end{subequations}

%The quantity $\displaystyle{\overline{\Delta t}=\frac{\Delta
%    t\,\varepsilon^{2}}{\varepsilon^{2}+\sigma\,\Delta t}}$, is
%homogeneous to
%$\Delta t$ and one can remark that we recover \eqref{eq:Hd} as soon
%as $\Delta t$ tends to zero.

Considering the scheme \eqref{scheme-JPT}-\eqref{scheme-JPT2} in
the limit of $\varepsilon$ to zero provides a numerical scheme approximating the solutions of the
convection-diffusion problem \eqref{limit_1}-\eqref{limit_2}. It reads:
\begin{align}
    & \bar u_i^{n+1}=\bar u_i^n-\frac{\Delta t}{2\Delta x}(\bar v_{i+1}^n-\bar v_{i-1}^n)+\frac{\lambda\Delta t}{2\Delta x}(\bar u_{i+1}^n-2\bar u_i^n+\bar u_{i-1}^n), \label{scheme-JPT-ubar}\\
    & \bar v_i^{n+1}=a\bar u_i^{n+1}-\frac{\lambda^2}{2\Delta x}(\bar u_{i+1}^{n+1}-\bar u_{i-1}^{n+1}). \label{scheme-JPT-vbar}
\end{align}

The numerical results presented in Figure \ref{fig:linear}-top correspond
to a Riemann initial condition
$u(0,x) = 2\cdot \mathbf 1_{x<0}+  \mathbf 1_{x\geq 0}$, the initial
condition $v(0,x)$ is set to local equilibrium, namely $v(0,x) =a
u(0,x)$.
The final time of computation is $T=0.1$, the domain $[0,1]$ is
discretized with $N_x=200$ cells and the CFL parameter is set to 0.95.
The wave speed is $\lambda=0.72$, $a=0.5$ and the relaxation parameter
$\varepsilon$ is set to 1, according
to the subcharacteristic condition \eqref{cond_sous_car}. Hence 
the profiles of $(u,v)$ present the two waves of speed $\pm \lambda$
and are strongly mollified by the relaxation term. On the other hand
the asymptotic limit $\bar u$ corresponds to the solution of the
parabolic equation \eqref{limit_1} and $\bar v$ to \eqref{limit_2}.

The convergence rate presented in Figure \ref{fig:linear}-bottom has been
computed with 
different values of $\varepsilon$ from $10^{-1}$ to $0.5\times
3.10^{-3}$ with the same initial condition and various space discretizations.
The log-log scale figure asses the convergence rate in
$O(\varepsilon^4)$, which is in agreement with theorems
\ref{thm:continuous} and \ref{thm:2}.

\section{Conclusion and prospects}

In this work, we establish the convergence of the solution of the linear Jin and
Xin model to a solution of the convection-diffusion equation
obtained in the limit $\varepsilon\to 0$ by a relative entropy
method.
The estimates give
a rate of convergence in $O(\varepsilon^4)$ , which is found
numerically.
Moreover, the
technique adapted to the discrete setting again gives the same rate of
convergence for a class of finite volume schemes that are discrete in space and
continuous in time.

Actually the estimate holds for the nonlinear case, considering a
relaxation term of type $f(u)-v$, with nonlinear function $f$. In this
case, relative entropy computations are more technical than in the
linear case. This is because the entropy of the Jin and Xin system
with nonlinear $f$ is not explicit, see \cite{Serre2000}. In fact, the relative
entropy identity contains terms that are more difficult to handle at
both continuous and discrete levels.

However the same rate of convergence in $\varepsilon^4$ can be
observed. As a numerical evidence, we present in Figure
\ref{fig:nonlinear} the numerical results for the Jin and Xin model
for $f(u)=u^2/2$. Figure \ref{fig:nonlinear}-top presents the profile of $(u,v)$
for the system \eqref{JinXin-1}-\eqref{JinXin-2}. The initial data is
at local equilibrium namely  $u(0,x) = 2\cdot \mathbf 1_{x<0}+  \mathbf
1_{x\geq 0}$ and $v(0,x)=f(u(0,x))$. One sets 
$\varepsilon=1$ and $\lambda=3$ in order to satisfy the subcharacteristic
condition. The others parameters are identical to the linear test case.
Figure \ref{fig:nonlinear}-bottom shows the behaviour of
$\phi$ with respect to $\varepsilon$ in log-log
scale.

A natural perspective is to adapt the relative entropy method to exhibit the
convergence rate in the nonlinear case. 

\begin{figure}
\centering
\includegraphics[width=0.75\linewidth]{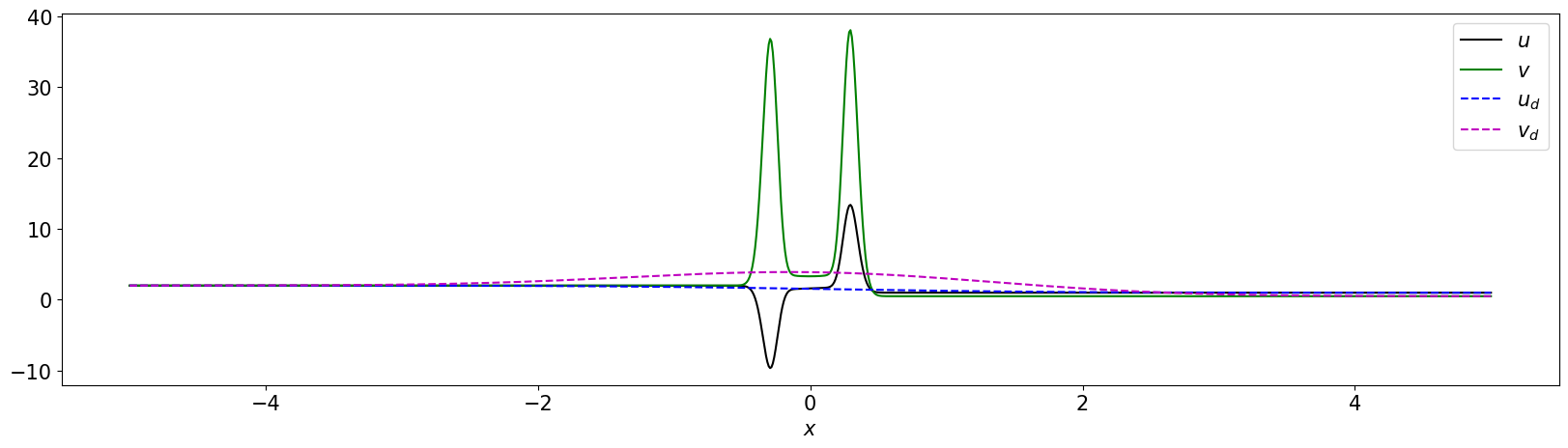}

\includegraphics[width=0.75\linewidth]{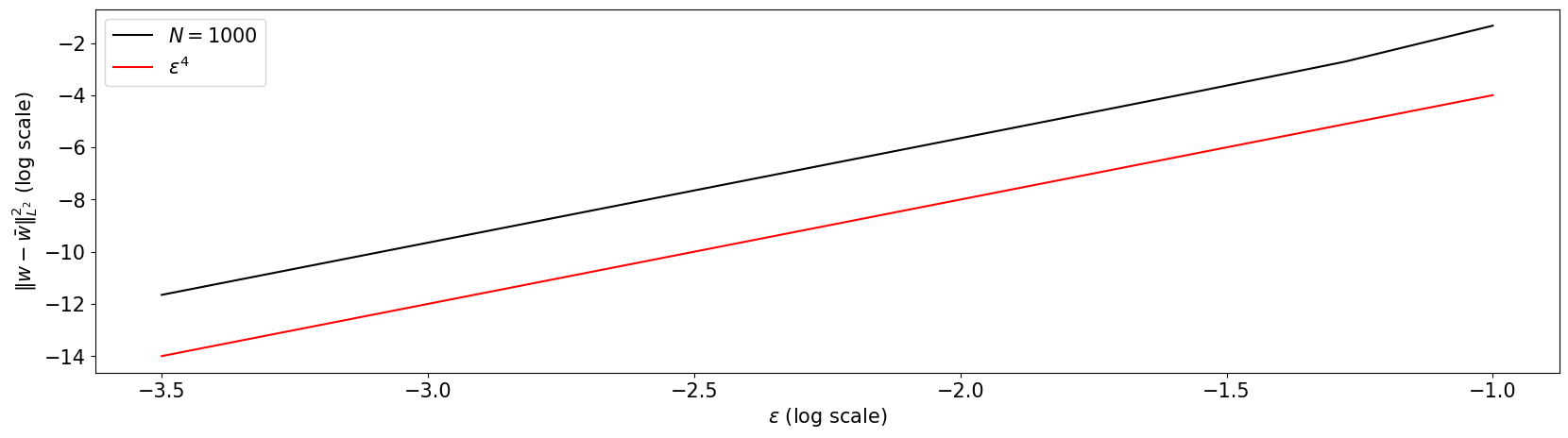}
\caption{Nonlinear test case. Top: profile of $(u,v)$ in space, compared to $(\bar u,\bar v)$.
  Bottom: $L^2$ norm of the error $\|(u,v)-(\bar u, \bar v)\|_{L^2(Q_T)}^2$ with respect to $\varepsilon$ in log scale.}
\label{fig:nonlinear}
\end{figure}

% ==============================================
% Biblio
%==============================================
\bibliographystyle{plainurl}
\bibliography{biblio_lim_diff}

\end{document}